\documentclass[12pt]{amsart}
\usepackage{amssymb}
\usepackage{amsmath}

\newtheorem{theorem}{Theorem}[section]

\newtheorem{corollary}[theorem]{Corollary}

\newtheorem{definition}[theorem]{Definition}

\newtheorem{lemma}[theorem]{Lemma}

\newtheorem{proposition}[theorem]{Proposition}
\newtheorem{remark}[theorem]{Remark}

\numberwithin{equation}{section}

\begin{document}

\title{A Note  On Conformal Ricci Flow}

\author{Peng Lu, Jie Qing$^\dag$, and Yu Zheng }

\begin{abstract}
In this note we study conformal Ricci flow introduced by Arthur Fischer in \cite{Fi}. We use DeTurck's trick to rewrite conformal Ricci flow as a strong parabolic-elliptic partial differential equations. Then we prove short time existences for conformal Ricci flow
on compact manifolds as well as on asymptotically flat manifolds. We show that Yamabe constant is monotonically increasing along conformal Ricci flow on compact manifolds. We also show that conformal Ricci flow is the gradient flow for the ADM mass on asymptotically flat manifolds.
\end{abstract}

\renewcommand{\subjclassname}{\textup{2000} Mathematics Subject Classification}

\address{Peng Lu, Department of Mathematics, University of Oregon, Eugene, OR 97403, USA} \email{penglu@uoregon.edu}

\address{Jie Qing, Department of Mathematics, University of California, Santa Cruz, CA 95064, USA} \email{qing@ucsc.edu}

\address{Yu Zheng, College of Mathematics, East China Normal University, Shanghai, PR China}
\email{zhyu@math.ecnu.edu.cn}

\thanks{$^\dag$ Research partially supported by NSF DMS-1005295 and CNSF 10728103.}

\maketitle

\section{Introduction}

Suppose that ${\rm M}^m$ is a smooth m-dimensional manifold and that $g_0$ is a Riemmanian metric on ${\rm M}^m$ with a constant scalar curvature $s_0$. The conformal Ricci flow on ${\rm M}^m$ is defined as follows:
\begin{equation}\label{eq crf orig-0}
\left\{\aligned
& \frac{\partial g }{\partial t} + 2 \left( Ric - \frac {s_0}m g\right)
= - 2p g  \text{ in ${\rm M}^m\times (0, T)$} \\
& s[g(t)] = s_0 \text{ in ${\rm M}^m\times [0, T)$}
\endaligned\right.
\end{equation}
for a family of metrics $g(t)$ with initial condition $g(0) =g_0$
and a family of functions $p=p(t)$ on ${\rm M}^m\times [0,T)$, where $s[g(t)]$ is the scalar curvature of the evolving metric $g(t)$.  Conformal Ricci flow \eqref{eq crf orig-0} was introduced by Arthur Fischer in \cite{Fi} as a modified Ricci flow that preserves the constant scalar curvature of the evolving metrics. Because the role of conformal geometry plays in maintaining scalar curvature constant such a modified Ricci flow was named as conformal Ricci flow in \cite{Fi}.  It was shown in \cite{Fi} that on compact manifolds conformal Ricci flow is equivalent to
\begin{equation}\label{eq crf orig}
\left\{\aligned
& \frac{\partial g}{\partial t} + 2 \left( Ric - \frac {s_0}m g\right)
= -2p g \text{ in ${\rm M}^m\times (0, T)$} \\
& (m-1)\Delta p + s_0 p = - |{\rm Ric} - \frac {s_0}m g|^2 \text{ in ${\rm M}^m\times [0, T)$}
\endaligned\right.
\end{equation}
with the initial condition $g(0) = g_0$. Based on the fact that conformal Ricci flow \eqref{eq crf orig} is of  parabolic-elliptic nature analogous to Navier-Stokes equations,  the function $p$ is named as conformal pressure function in \cite{Fi}. Using the theory of dynamical systems on infinite dimensional manifolds it was proved in \cite{Fi} that conformal Ricci flow exists at least for short time on compact manifolds of negative Yamabe type. It was also observed in \cite{Fi} that Yamabe constant monotonically increases along conformal Ricci flow on compact manifolds of negative Yamabe type. Therefore it is hopeful that conformal Ricci flow does a good job in constructing Einstein metrics, considering the behavior of Hilbert-Einstein action on the space of Riemannian metrics.

In this note we adopt DeTurck's trick \cite{DT2}\cite{DT3} to eliminate the degeneracy of \eqref{eq crf orig} from the symmetry of diffeomorphisms and consider the following what we call DeTurck conformal Ricci flow
\begin{equation}\label{dcrf}
\left\{\aligned &
\frac{\partial}{\partial t} g  + 2 \left ( {\rm Ric} -\frac {s_0}m g\right )
= -2 p g + \mathcal {L}_{W} g\\
& (m-1) \Delta p + s_0 p = - | {\rm Ric} -\frac {s_0}m g |^2
\endaligned\right.
\end{equation}
for an appropriately chosen vector field $W$ (cf. \eqref{vector field}) with an initial metric $g(0) = g_0$ of constant scalar curvature $s_0$.  \eqref{dcrf} is a strong parabolic-elliptic partial differential equations. We use the contractive mapping theorem to prove the isomorphism property for linearized DeTurck conformal Ricci flow and we use an implicit function theorem to prove short time existence for DeTurck conformal Ricci flow. Therefore we obtain short time existence for confomral Ricci flow based on the discussion in Section \ref{subsec Dconformal Ricci flow}. For parabolic H\"{o}lder spaces and the theory of linear and nonlinear parabolic equations we take references mostly from \cite{Lu}.

\begin{theorem} Let $({\rm M}^m, \ g_0)$ be a compact Riemannian manifold of constant scalar curvature $s_0$ with no boundary. And suppose that  the elliptic operator $(m-1)\Delta[g_0] + s_0$ is invertible. Then there exists a small positive number $T$ such that conformal Ricci flow $g(t)$ with the initial metric $g_0$ exists for $t\in [0, T]$.
\end{theorem}

This extends the existence result in \cite{Fi} to include compact manifolds of positive Yamabe type. We also extend the monotonicity of Yamabe constant in \cite{Fi} as follows:

\begin{theorem}
Let $({\rm M}^m,g_0)$ be a compact Riemannian manifold with no boundary and let $g(t), \, t \in [0,T)$ be the solution of conformal Ricci flow with $g(0)=g_0$.
Suppose that $g_0$ is the only Yamabe metric in the conformal class $[g_0]$ with $s[g_0]=s_0$ and that $(m-1) \Delta[ g_0] +s_0$ is invertible.
Then there is $T_0 \in (0, T]$ such that each metric $g(t)$, $t\in [0, T_0)$, is a Yamabe metric and the Yamabe constant $Y[g(t)]$ increases strictly for $t \in [0,T_0)$ unless $g_0$ is an Einstein metric.
\end{theorem}

This theorem indicates that conformal Ricci flow is somehow better family of constant scalar curvature metrics than those obtained in \cite{Ko}.

On asymptotically flat manifolds we use weighted H\"{o}lder spaces defined in \cite{LP} and define weighted parabolic H\"{o}lder spaces based on the similar ones in \cite{Lu} \cite{OW}.

\begin{theorem}\label{existence AF} Let $({\rm M}^m, \ g_0)$ be scalar flat and 
asymptotically flat manifold with $g_0 - g_e\in C^{4, \alpha}_{-\tau}$ and $\tau \in (0, m-2)$, 
where $g_e$ is the standard Euclidean metric.  Then there exists a small positive number $T$ such that the conformal Ricci flow $g(t)$  from the initial metric $g_0$ exists for $t\in [0, T]$ and $g(t)-g_e \in C^{1, 2+\alpha}_{-\tau}([0, T]\times{\rm M})$.
\end{theorem}

It is easily seen that \eqref{eq crf orig-0} and \eqref{eq crf orig} are equivalent on asymptotically flat manifolds because of the uniqueness of bounded solutions to linear parabolic equations on asymptotically flat manifolds. The scalar flat assumption in Theorem \ref{existence AF}  is less stringent than it looks. Thanks to \cite[Lemma 3.3, Corollary 3.1]{SY} we know that one can always conformally deform  an asymptotically flat metric with nonnegative scalar curvature into a scalar flat asymptotically flat metric.

Conformal Ricci flow is the gradient flow for the ADM mass on asymptotically flat manifolds in the following sense.

\begin{theorem}\label{monotone ADM mass}
Let  $g(t)$ be the conformal Ricci flow obtained in Theorem \ref{existence AF} for $\tau\in (\frac {m-2}2, m-2)$.  Then
\[
\frac{d}{dt} m(g(t)) = -2 \int_M |Ric[g(t)] |^2 d\text{vol}[g(t)].
\]
In particular, the ADM mass $m(g(t))$ is strictly decreasing under conformal Ricci flow except that $g_0$ is
the Euclidean metric.
\end{theorem}
As a quick application of Theorem \ref{monotone ADM mass} one can easily show the rigidity part of the celebrated positive mass theorem of Schoen and Yau \cite{SY}. The monotonicity of the ADM mass along conformal Ricci flow is sharply in contrast to the invariance of the ADM mass \cite{DM} \cite{OW} along Ricci flow on asymptotically flat manifolds.

The organization of the note is as follows: in Section 2 we introduce conformal Ricci flow and
establish the monotonicity of Yamabe constant on compact manifolds. In Section 3 we prove short existences for conformal Ricci flow both on compact manifolds and on asymptotically flat manifolds. In Section 4 we recall the definition of the ADM mass and show that conformal Ricci flow on asymptotically flat manifolds is the gradient flow for the ADM mass.

\vskip 0.1in\noindent{\bf Acknowledgment}
The authors would like to mention that Lars Andersson and Dan Lee have communicated to us that they are also aware of some version of Theorem \ref{monotone ADM mass}. The authors are gateful that Dan Lee brought  a related paper \cite{RH} to their attention. Part of this work was done while the authors were visiting Beijing International Center for Mathematical Research and East China Normal University. The authors are grateful to them for the hospitality.


\section{Conformal Ricci Flow}

In this section we first introduce  conformal Ricci flow and calculate evolution equations for curvatures along conformal Ricci flow. We then discuss the monotonicity of Yamabe quotients and Yamabe constants along conformal Ricci flow.

\subsection{Conformal Ricci Flow}\label{intr crf}

Suppose that ${\rm M}^m$ is a smooth m-dimensional manifold and that $g_0$ is a Riemmanian metric on ${\rm M}^m$ with constant scalar curvature $s_0$. In \cite{Fi},  the conformal Ricci flow on ${\rm M}^m$ is defined by \eqref{eq crf orig-0} for a family of metrics $g(t)$ with initial condition $g(0) =g_0$
and  a family of functions $p=p(t)$ on ${\rm M}^m\times [0,T)$.

As calculated in \cite{Fi}, the normalization condition $s[g(t)] = s_0$ may be replaced by an elliptic equation and rewrite \eqref{eq crf orig-0} as \eqref{eq crf orig}. The equivalence between \eqref{eq crf orig-0} and \eqref{eq crf orig} was proved in \cite{Fi} in cases when ${\rm M}^m$ is a compact manifold without boundary (cf. \cite[Proposition 3.2, Proposition 3.4]{Fi}).  Based on the evolution equation for scalar curvature, it is easily seen that \eqref{eq crf orig-0} always implies \eqref{eq crf orig} and \eqref{eq crf orig} implies \eqref{eq crf orig-0} when the solution to the linear heat equations is unique,  which is true in both cases we consider in this note: compact and asymptotically flat.

One important issue for geometric PDE is the scaling property. It is easily seen that
\begin{equation}
g_\lambda (\cdot, t))  = \lambda^{-2}g(\cdot, \lambda^2 t)) \quad \text{and} \quad p_\lambda(\cdot, t) = \lambda^2p(\cdot, \lambda^2t)
\end{equation}
solve conformal Ricci flow \eqref{eq crf orig} if $(g, p)$ solve conformal Ricci flow \eqref{eq crf orig}.


\subsection{Curvature Evolution Equations under conformal Ricci flow} \label{subsec cur evol conformal Ricci flow}

To understand conformal Ricci flow one often needs to calculate how curvatures behave along conformal Ricci flow. The calculations are straightforward. Suppose that we consider a general geometric flow
\begin{equation}\label{general flow}
\frac \partial{\partial t} g = -2T.
\end{equation}
Then we recall the evolution equations for curvarures (cf. \cite{CLN} \cite{Be}, for example)
$$
\aligned
\frac \partial{\partial t} s  & = 2\Delta \Theta - 2\nabla^i\nabla^j T_{ij} + 2 R^{ij}T_{ij}\\
\frac \partial{\partial t} R_{ij} & = \Delta T_{ij} - \nabla_i\nabla^k T_{kj} - \nabla_j\nabla^k T_{ki} + \nabla_i\nabla_j \Theta + 2 R_{ikjl}T^{kl} - R_{ik}T^k_{\ \, j} - R_{jk}T^k_{\ \,  i}\\
\frac \partial{\partial t} R_{ikjl} & = \nabla_i\nabla_j T_{kl} - \nabla_i\nabla_lT_{kj} - \nabla_k\nabla_jT_{il}  + \nabla_k\nabla_lT_{ij} - R_{ikjm}T^m_{\ \ l} - R_{ikml}T^m_{\ \ j}
\endaligned
$$
where $\Theta := g^{ij}T_{ij}$. Particularly for conformal Ricci flow, where
$$
T = {\rm Ric} - \frac {s_0}m g + pg \ \text{and} \ \Theta = s - s_0 + mp,
$$
we may calculate the evolution equations for curvatures under conformal Ricci flow as follows:
$$
\aligned
\frac \partial{\partial t} s  & = \Delta s  + 2\frac {s_0}m (s - s_0) + 2p (s -s_0) +2(m-1)\Delta p + 2s_0p + 2|{\rm Ric} - \frac {s_0}m g|^2\\
\frac \partial{\partial t} R_{ij} & = \Delta R_{ij}  + 2 R_{ikjl}R^{kl} - 2R_{ik}R^k_{\ \, j} + (m-2) \nabla_i\nabla_j p + \Delta p g_{ij}\\
\frac \partial{\partial t}{\rm Rm} & = \Delta {\rm Rm} + {\rm Rm}\ast{\rm Rm} + {\rm Ric}\ast{\rm Rm} + 2\frac {s_0}m{\rm Rm} - 2p{\rm Rm} + T(\nabla^2 p)
\endaligned
$$
where operations $\ast$ stands for contractions of tensors and
$$
T(\nabla^2 p)_{ikjl} = g_{kl}\nabla_i\nabla_j p - g_{kj}\nabla_i\nabla_l p- g_{il}\nabla_k\nabla_j p + g_{ij}\nabla_k\nabla_l p.
$$


\subsection{Yamabe Constants under Conformal Ricci Flow} \label{sect Yamabe constant under crf}

On compact manifolds with no boundary, along conformal Ricci flow, we may calculate
$$
\Theta = mp
$$
$$
\frac{\partial}{\partial t} d\text{vol}[g(t)] = - mpd\text{vol}[g(t)]
$$
\begin{equation}\label{eq volume}
\frac d{d t} \text{vol} ({\rm M}) = - m \int_{\rm M} p d\text{vol}[g] = \frac m {s_0}\int_{\rm M} |{\rm Ric} - \frac {s_0}m g|^2 d\text{vol}[g]
\end{equation}

Given a compact Riemannian manifold $({\rm M}^m,h)$ with no boundary, the Yamabe quotient is defined as:
$$
Q[h] := \frac {\int_{\rm M} s[h] d\text{vol}[h]}{\text{vol}[h]({\rm M})^{\frac{m-2}{m}}}.
$$
and the Yamabe constant is defined as:
$$
Y[h] = \inf_{h\in [h]}Q[h].
$$
A Riemannian metric $h$ is said to be a Yamabe metric if and only if
$$
Q[h] = Y[h].
$$
Thus from \eqref{eq volume} {we have

\begin{proposition} \label{Q CFT} Suppose that $g(t)$, $t\in [0, T)$, is a solution to conformal Ricci flow \eqref{eq crf orig} on a compact manifold with no boundary with $s[g_0]\neq 0$. Then
$Q[g(t)]$ increases strictly unless $g_0$ is an Einstein metric.
\end{proposition}

\vskip .1cm
Next we consider the change of Yamabe constants along conformal Ricci flow. As observed in  \cite{WZ}  \cite{ChL} \cite{And} \cite{Ko} Yamabe constant could behave rather irregular in general among the metrics of positive Yamabe type.

\begin{theorem}\label{thm Yamabe constant monotone}
Let $({\rm M}^m, \ g_0)$ be a compact Riemannian manifold with no boundary and let $g(t), \, t \in [0,T)$ be the solution of conformal Ricci flow with $g(0)=g_0$.
Suppose that $g_0$ is the only Yamabe metric in the conformal class $[g_0]$ with $s[g_0]= s_0$ and that $(m-1) \Delta[g_0] + s_0$ is invertible.
Then there is $T_0 \in (0, T]$ such that each metric $g(t)$, $t\in [0, T_0)$, is a Yamabe metric and the Yamabe constant $Y[g(t)]$ increases strictly for $t \in [0,T_0)$ unless $g_0$ is an Einstein metric.
\end{theorem}

\begin{proof} We use the proof by contradiction. Assume otherwise there is a sequence
$t_i \rightarrow 0^+$ such that $g(t_i)$ are not Yamabe metrics.
Let $\tilde{g}_i$ be a Yamabe metric in the conformal class $[g(t_i)]$.
Then by the compactness of Yamabe metrics  $\tilde{g}_i$
converges to a Yamabe metric $g_\infty \in [g_0]$ (taking subsequence if necessary).
By the assumption that $g_0$ is the only Yamabe metric in $[g_0]$ we have $g_\infty =g_0$. That is to say that $g(t_i)$ and $\tilde g(t_i)$ both converge to $g_0$, which is a contradiction since $(m-1) \Delta[ g_0] + s_0$ is assumed to be invertible. The similar arguments have been used in \cite{WZ} \cite{ChL} \cite{And} \cite{Ko}.
\end{proof}


\section{Short Time Existences of conformal Ricci flow}

In this section we prove the short time existence of conformal Ricci flow. The first step is to combine the two equations in conformal Ricci flow into one evolution eqution with one non-local term. Hence \eqref{eq crf orig} turns into
\begin{equation}\label{eq crf-1}
\frac\partial{\partial t}g  + 2({\rm Ric} - \frac{s_0}m g) = -2\mathcal {P}(g)g \text{ on ${\rm M}$},
\end{equation}
where
$$
\mathcal {P}(g) = ((m-1)\Delta +s_0)^{-1}|{\rm Ric} -\frac {s_0}mg|^2,
$$
provided that $(m-1)\Delta[g(t)] + s_0$ is invertible for all $t\in [0, T]$. The strategy to prove the short time existence for conformal Ricci flows is similar to the one used in \cite{DT2} \cite{DT3} to prove the short time existence for Ricci flow. We will first prove the short time existence for DeTurck conformal Ricci flow written in one equation:
\begin{equation}\label{eq dtcrf-1}
\frac\partial{\partial t}g  + 2({\rm Ric} - \frac{s_0}m g) = -2\mathcal {P}(g)g + \mathcal {L}_{W}g \text{ on ${\rm M}$}.
\end{equation}
To prove the short time existence for \eqref{eq dtcrf-1} we calculate the linearization of DeTurck conformal Ricci flow and apply an implicit function theorem.


\subsection{DeTurck's Trick} \label{subsec Dconformal Ricci flow}

The conformal Ricci flow as a system of differential equations is of parabolic-elliptic nature similar to Navier-Stokes equations. The significant difference between conformal Ricci flow and Navier-Stokes equations is that conformal Ricci flow is a geometric flow. Hence we need to find ways to eliminate the degeneracy of conformal Ricci flow arising from the symmetry of diffeomorphisms.

In this subsection we will follow the idea from the improved version \cite{DT3} of the approach to solve the short time existence for the Ricci flow in \cite{DT2} to rid off the degeneracy of diffeomorphisms for conformal Ricci flow. To introduce DeTurck's trick we first recall the following operator $G$. Let $g$ be a Riemannian metric on ${\rm M}^m$. The operator $G$ on symmetric 2-tensor $B$ is defined as:
\begin{equation}\label{op g}
G(B) :=  B-\frac{1}{2} \text{Tr}_g (B) g.
\end{equation}
Also recall that the divergence operator $\delta$
$$
(\delta B)_i := \nabla^j B_{ij}: \Gamma(S^2({\rm M}^m))\to \Gamma (T^*{\rm M}^m)
$$
and its adjoint operator $\delta^*$
$$
(\delta^* \omega)_{ij} := - \frac{1}{2} (\omega_{i,j} + \omega_{j, i}): \Gamma(T^*{\rm M}^m) \to
\Gamma(S^2({\rm M}^m)).
$$
Note that, if $X$ is the dual vector field of $\omega$, then
$\delta^* \omega = - \mathcal {L}_{X} g$ where $\mathcal {L}_X$ denotes the Lie derivative in $X$ direction.

According to DeTurck's improved version \cite{DT3} of his approach to the short time existence
of Ricci flow \cite{DT2} we consider the following gauge-fixed conformal Ricci flow on ${\rm M}^m$
\begin{equation}\label{eq dtcrf-0}
\left\{\aligned
& \frac{\partial}{\partial t} g + 2 \left ({\rm Ric} - \frac {s_0}m
g \right ) = - 2 p g + 2(\delta^*( \tilde g^{-1} \delta G(\tilde g)))  \\
& (m-1) \Delta p + s_0 p = - | {\rm Ric} -\frac {s_0}m g |^2
\endaligned\right.
\end{equation}
for a family of metrics $g(t)$ with $g(0)=g_0$ and a family of functions $p(t)$ on ${\rm M}^m\times[0,T)$, where $\tilde g$ is any fixed metric on ${\rm M}^m$.

Suppose that $g(t)$, $t\in [0, T)$, solves \eqref{eq dtcrf-0}. Then we consider the time dependent vector field $W$
\begin{equation}\label{vector field}
W^k := g^{ij}  \left ( \Gamma_{ij}^k[g] - \Gamma_{ij}^k[\tilde{g}]
\right ).
\end{equation}
It turns out that  (cf. \cite[\S 6]{Ha2} \cite {Shi})
$$
2 (\delta^*(\tilde g^{-1} \delta G(\tilde g)))= \mathcal {L}_{W} g.
$$
Hence we can rewrite \eqref{eq dtcrf-0} as follow:
\begin{equation}\label{eq dtcrf}
\left\{\aligned &
\frac{\partial}{\partial t} g  + 2 \left ( {\rm Ric} -\frac {s_0}m g\right )
= -2 p g + \mathcal {L}_{W} g\\
& (m-1) \Delta p + s_0 p = - | {\rm Ric} -\frac {s_0}m g |^2 .
\endaligned\right.
\end{equation}

Conformal Ricci flow \eqref{eq crf orig} and DeTurck conformal Ricci flow \eqref{eq dtcrf} are related to each other by coordinate changes in the following sense. Suppose that $\hat g(t)$ solves DeTurck conformal Ricci flow equations \eqref{eq dtcrf} and $W$ is given as in \eqref{vector field}. We consider the one-parameter family of diffeomofphisms $\varphi_t$ generated by $W$ on ${\rm M}^m$ as:
\begin{equation}\label{ode}
\frac{\partial }{\partial t} \varphi_t(x) = - {W}(\varphi_t(x),t), \qquad
\qquad \varphi_0(x) =x
\end{equation}
for some time period $[0, T)$.
\begin{lemma}\label{lem dcrf to sol crf}
Let $(\hat{g}(t), \hat{p}(t)), \, t \in [0,T)$, be a solution to DeTurck conformal Ricci flow \eqref{eq dtcrf} on manifold $M^m$ with the initial metric $g_0$. Assume that the solution $\varphi_t(x)$ to \eqref{ode} exists for $t \in [0,T)$. Let
$$
g(t) := \varphi^*_t \hat{g}(t) \text{ and } p(t) := \hat{p}( \varphi_t(x),t).
$$
Then $(g(t),p(t)),  \, t \in [0,T)$, is a solution to conformal Ricci flow \eqref{eq crf orig}
on manifold ${\rm M}^m$ with $g(0) =g_0$.
\end{lemma}

\begin{proof}
We simply compute by using \eqref{eq dtcrf} (cf.  \cite[\S 2.6]{CLN}, for example)
\begin{align*}
\frac{\partial}{\partial t} g(t) &= \varphi_t^* \left (\frac{\partial}{\partial t}
\hat{g}(t) \right)   + \left . \frac{\partial}{\partial s} \right \vert_{s=0} \left
 ( \varphi_{t+s}^* \hat{g}(t) \right)  \\
&= -2  \varphi_t^* \left ( {\rm Ric}[\hat{g}] - \frac {s_0}m {\hat{g}}
+ \hat{p} {\hat{g}}  \right) + \varphi_t^* \left (\mathcal {L}_{{W}} \hat{g} \right)
- \mathcal {L}_{\left( \varphi_t^{-1} \right )_* {W}} (\varphi_t^* \hat{g} )  \\
& = -2 \left ( {\rm Ric}[g] -\frac {s_0}m g + p g \right).
\end{align*}
The second equation for $p$ in \eqref{eq crf orig} is readily seen to be true since the scalar curvature under both flows is kept constant as $s_0$.
\end{proof}

This lemma is particularly important to us because it enables us to prove the short time
existence of conformal Ricci flow by proving the short time existence of DeTurck conformal Ricci flow. The later will be proven to be a system of parabolic-elliptic equations (cf. Lemma \ref{linear-dtcrf}).

On the other hand, suppose that $(g(t), p(t))$, $t\in [0, T)$, solves conformal Ricci flow \eqref{eq crf orig} on ${\rm M}^m$ with initial metric $g_0$ and $\tilde g$ is any fixed metric on ${\rm M}^m$. We then consider the harmonic map flow
\begin{equation}\label{h-flow}
\frac{\partial }{\partial t} \varphi_t = \Delta_{g(t), \tilde{g}} \varphi_t,
\qquad \qquad  \varphi_0 = \operatorname{Id}
\end{equation}
for $\varphi_t: {\rm M}^m\to {\rm M}^m$, where the nonlinear Laplacian
in local coordinates is given as
$$
\left ( \Delta_{g_1, g_2}f \right )^\gamma = \Delta [g_1] f^\gamma +
\Gamma_{\alpha \beta}^\gamma [g_2] \frac{\partial f^\alpha}
{\partial x^i}  \frac{\partial f^\beta} {\partial x^j}  g_1^{ij} .
$$
The following lemma is useful to derive the uniqueness of conformal Ricci flow via the uniqueness of DeTurck conformal Ricci flow. It is therefore readily seen that the uniqueness of conformal Ricci flow with a given initial metric holds at least on compact manifolds with no boundary.

\begin{lemma} \label{lem crf to sol dcrf}
Let $(g(t), p(t)), \, t \in [0,T)$, be a solution to conformal Ricci flow (\ref{eq crf orig})
on manifold ${\rm M}^m$ with initial metric $g_0$. Assume the solution $\varphi_t:M
\rightarrow M$ to the harmonic map flow \eqref{h-flow} exists for $t \in [0,T)$. Let
$$
\hat{g}(t) :=  (\varphi_t^{-1}  )^* g(t) \text{ and } \hat{p}(x,t) := p(\varphi_t^{-1}(x) ,t).
$$
Then $(\hat{g}(t), \hat{p}(t)), t \in [0,T)$, is a solution to DeTurck conformal Ricci flow \eqref{eq dtcrf}  on manifold ${\rm M}^m$ with initial metric $g_0$.
\end{lemma}

\begin{proof}
This follows from a calculation similar to the one in the proof of Lemma \ref{lem dcrf to sol crf} after identifying the vector field ${W} =\Delta_{g(t), \tilde{g}} \varphi_t$(cf. \cite[p.117]{CLN}, for example).
\end{proof}


\subsection{The linearization of DeTurck conformal Ricci flow}\label{sec linear cal}

In this subsection we compute the linearization of DeTurck confomral Ricci flow \eqref{eq dtcrf-1}.
To do so we set
$$
g_\lambda(t) = g(t)+\lambda h(t)
$$
for a family of symmetric 2-tensors $h(t)$ and for $\lambda\in (-\epsilon, \epsilon)$. We rewrite DeTurck conformal Ricci flow
\begin{equation}\label{set-up}
\mathcal {M}(g(t)) = \frac{\partial}{\partial t} g  + 2 \left ( {\rm Ric} -\frac {s_0}m g\right )
+ 2 \mathcal {P}(g) g -  \mathcal {L}_{W} g := \frac{\partial}{\partial t} g - \mathcal {F}(g(t))=0
\end{equation}
and calculate
$$
\frac d{d\lambda}|_{\lambda =0} \mathcal {M}(g_\lambda).
$$

To compute the linearization of $\mathcal {P}$ we first calculate
$$
\frac{d}{d\lambda} |_{\lambda=0}   \Delta [g_\lambda] \mathcal {P}(g_\lambda)
=  -  h_{ij}  \nabla^i \nabla^j \mathcal {P} - \frac 12 (2 \nabla^i
h_{ij} - \nabla_j h^i_{\ \, i} ) \nabla^j \mathcal {P}+ \Delta \mathcal {P}'
$$
 (cf, \cite[(S.5) on p.547]{CLN}, for example), where $\mathcal{P}' = \frac d{d\lambda}|_{\lambda=o}\mathcal{P}(g_\lambda)$. Next we may calculate
$$
\frac{d}{d\lambda} |_{\lambda=0} | {\rm Ric}[g_\lambda] - \frac {s_0}m g_\lambda |^2
$$
using the linearization of Ricci curvature (cf. \cite{Be} \cite{CLN})
\begin{equation}\label{lichnerowicz}
\aligned
\frac{d}{d\lambda} |_{\lambda=0} 2R_{ij}[g_\lambda] & =  -
\Delta h_{ij} - 2R_{ikjl} h_{kl} + R_{ik} h^k_{\ \, j} +
R_{jk}h^k_{\ \, i} \\ & - \nabla_i \nabla_j  h^k_{\ \, k} +
\nabla_i \nabla^k h_{kj} +  \nabla_j \nabla^k h_{ki} \endaligned
\end{equation}
 (see (2.31) in \cite{CLN}, for example).  In summary we have
$$
(m-1)\Delta\mathcal {P}' + s_0 \mathcal {P}'  - P^{ijkl}_1 \nabla_i\nabla_j h_{kl} -
P^{ijk}_2 \nabla_i h_{jk} + P_3^{ij}h_{ij} =0,
$$
that is
\begin{equation}\label{linear-P}
\mathcal {P}' = ((m-1)\Delta +s_0)^{-1}(P_1\ast\nabla^2 h + P_2\ast\nabla h +P_3\ast h),
\end{equation}
where $P_1, P_2, P_3$ are tensors that depend on curvature of $g(t)$ and up to second order spatial derivatives of $p$.

In the calculation of the linearization of $\mathcal{M}$ the crucial step is to calculate
\begin{align*}
\frac{d}{d\lambda} \vert_{\lambda=0} \mathcal {L}_{W_\lambda} g_\lambda = - \frac{d}{d\lambda}  \vert_{\lambda=0 } \delta^* [g_\lambda] \omega_\lambda,
\end{align*}
where $(\omega_\lambda)_i := (g_\lambda)_{ik} W_\lambda^k$ and
$W_\lambda =  -\tilde{g}^{-1} \delta[g_\lambda] G[g_\lambda](\tilde{g})$. In fact the key point of the DeTurck's trick is to collect the part of second order covariant derivatives of $h$
in the above and realize that it cancels the second line in \eqref{lichnerowicz}. To see that we first collect terms involving the first order covariant derivatives of $h$ in the fiollowing
\begin{equation*}
\frac{d}{d\lambda}\vert_{\lambda=0} (\omega_\lambda)_i =
\nabla^k h_{ki} - \frac 12 \nabla_i h^k_{\ \, k}  + \text{ other terms } \cdots.
\end{equation*}
Then we collect the second order covariant derivatives of $h$ in the following
\begin{equation}\label{eq 2nd term in linear DT operator}
 \frac{d}{d \lambda}|_{\lambda=0} ( (\delta_{g_\lambda})^*\omega_\lambda)_{ij}
=  - \nabla_i\nabla^k h_{ki} - \nabla_j\nabla^k h_{kj} + \nabla_i\nabla_j h^k_{\ \, k} + \text{other terms}.
\end{equation}
Therefore
\begin{equation}\label{linear-M}
\frac d{d\lambda}|_{\lambda=0}\mathcal {M} (g_\lambda)= \frac \partial{\partial t} h - \Delta h + 2\mathcal {P}' g + M_1^{ijk}\nabla_i h_{jk} +M_2^{ij}h_{ij},
\end{equation}
where $M_1$ depends only $g(t)$ and $\mathcal {P}(g)$ and $M_2$ depends on the curvature of $g(t)$ and $\mathcal {P}(g)$. To summarize we have

\begin{lemma}\label{linear-dtcrf} Suppose that $g(t)$, $t\in [0, T]$,  is a family of metrics such that the elliptic operator $(m-1)\Delta[g(t)] +s_0$ is invertible for all $t\in [0, T]$. Then the linearization of DeTurck conformal Ricci flow equations \eqref{eq dtcrf-1} at the metrics $g(t)$ in the directions of the symmetric 2-tensors $h(t)$ is given as
\begin{equation}\label{eq linear dtcrf-1}
D\mathcal {M}(g)(h) = \frac \partial{\partial t} h - \Delta h + 2\mathcal {P}' g + M_1\ast \nabla h + M_2\ast h
\end{equation}
where
$$
\mathcal {P}' = ((m-1)\Delta+s_0)^{-1}(P_1\ast\nabla^2 h+P_2\ast\nabla h +P_3\ast h),
$$
$P_1, P_2, P_3$ are tensors depending on curvature of $g(t)$ and up to the second order derivatives in spatial variables of $\mathcal {P}(g)$, and $M_1, M_2$ are tensors depending on curvature of $g(t)$ and function $\mathcal {P}(g)$.
\end{lemma}


\subsection{Short Time Existence on Closed Manifolds} \label{subsec Holder space def}

Let us first solve conformal Ricci flow on a compact manifold ${\rm M}^m$ with no boundary. There are many books that are good for references in linear and nonlinar systems of parabolic equations. We will mostly use the book \cite[\S 5.1]{Lu}, in particular Theorem 5.1.21 in \cite{Lu} for existence and standard estimates. We adopt definitions of parabolic H\"{o}lder spaces from \cite[p. 175-177]{Lu}.  We use the same notations for parabolic H\"{o}lder spaces for functions and tensor fields when there is no confusion. To define the norms for tensor fields we may use the initial metric and local coordinate charts.


\subsubsection{Preliminaries}

Since we deal with systems of parabolic-elliptic equations we need to consider elliptic estimates with time parameter. There is an advantage to use only the super norm in time variable, as indicated by the following lemma.
\begin{lemma} \label{lem f c 0 parameter dependence}
Let $g(t), \, t \in [0,T]$, be a family of smooth Riemannian metric on
compact manifold ${\rm M}^m$ with no boundary.
Suppose operator $(m-1) \Delta[g(t)]+s_0$ is invertible for $t\in [0, T]$.
Then the equation
\begin{equation}\label{eq linearization of p 1 short}
(m-1) \Delta[g(t)] p(t) + s_0 p(t) = \gamma
\end{equation}
has a unique solution $p \in C^{0, 2+\alpha}$ for each $\gamma \in C^{0, \alpha}$. Moreover $p$ satisfies the estimate
\begin{equation}\label{eq ellp f est Laplace 2mc}
\|p\|_{C^{0, 2 + \alpha}} \leq C \| \gamma \|_{C^{0, \alpha}}
\end{equation}
for some constant $C$ independent of $\gamma$.
\end{lemma}

\begin{proof} In the light of standard Schauder estimates for elliptic PDE we only need to varify that $p(t)$ is continuous in time variable, which is a consequence of classical Bernstein estimates.
\end{proof}

The following interpolatory inclusion will be useful in the proof of the short-time existences (cf.  \cite[Lemma 5.1.1]{Lu}).

\begin{lemma} \label{lem interpolation} There is a constant $C$
independent of $T$ such that for any $t_1, t_2 \in [0,T]$ we have
\[
\|h(t_1, \cdot) - h(t_2, \cdot)\|_{C^{k-2, \alpha}} \leq C
\cdot |t_1 -t_2 |\cdot \| h\|_{C^{1, k+\alpha}}.
\]
for all $h \in C^{1, k + \alpha}([0, T]\times {\rm M})$.
\end{lemma}


\subsubsection{On Linearized DeTurck conformal Ricci Flow}\label{on linearized Dconformal Ricci flow}

We first solve the linarized DeTurck conformal Ricci flow
\begin{equation}\label{eq linear dtcrf-2}
\left\{\aligned D\mathcal {M}(g)(h) & = \frac \partial{\partial t} h - \Delta h + 2\mathcal {P}' g + M_1\ast \nabla h + M_2\ast h = \gamma\\
h (0, \cdot) & = 0
\endaligned\right.
\end{equation}
for appropriately given metrics $g(t)$ for each $\gamma \in C^{0, \alpha}$. Namely,

\begin{proposition} \label{lem 4}
 Suppose that $g(t)$, $t\in [0, T]$,  is a family of metrics such that the elliptic operator $(m-1)\Delta[g(t)] +s_0$ is invertible for all $t\in [0, T]$.  Then, for $\gamma \in C^{0, \alpha}$, the initial value problem for \eqref{eq linear dtcrf-2} has a unique solution $h \in C^{1, 2+\alpha}$. Moreover
\begin{equation}\label{est of sol in haf}
 \|h\|_{C^{1, 2+\alpha}([0, T]\times{\rm M})}
\leq  C \| \gamma \|_{ C^{0, \alpha}([0, T]\times{\rm M})}.
\end{equation}
 \end{proposition}

\begin{proof}  To use contractive mapping type argument we consider the Banach space
$$
E_1([0,T^*]) = \{h\in C^{0, 2+\alpha}: h(0, \cdot) = 0\}.
$$
Given a $\tilde{h} \in E_1([0, T^*])$, based on Theorem 5.1.21 in \cite{Lu}, we first solve an usual system of linear parabolic equations
\begin{equation}\label{linearized Dconformal Ricci flow}
\left\{
\aligned
& \frac \partial{\partial t} h - \Delta h + M_1\ast \nabla h + M_2\ast h = \tilde\gamma  \\
& h(0, \cdot) =0,
\endaligned\right.
\end{equation}
where $\tilde\gamma= \gamma - 2{\mathcal P}'(\tilde h) g \in C^{0, \alpha}$
and ${\mathcal P}'(\tilde{h})$ is defined by \eqref{linear-P}. We remark here that it takes some work to extend Theorem 5.1.21 in \cite{Lu} to be applicable to our context, but there is no significant
issues in doing so.  Hence we may define a map
\[
 \Psi: E_1([0,T^*]) \rightarrow E_1([0,T^*])
 \]
by $\Psi (\tilde{h}) =h$. Note that, if set
$$
v = \Psi(\tilde h_1) - \Psi(\tilde h_2),
$$
then $v$ satisfies
$$
\left\{
\aligned
& \frac \partial{\partial t} v - \Delta v + M_1\ast \nabla v + M_2\ast v = 2( {\mathcal P}'(\tilde h_2) -{\mathcal P}'(\tilde h_1))g \\
& v(0, \cdot ) =0,
\endaligned\right.
$$
Due to the fact that
$$
\|( {\mathcal P}'(\tilde h_2) -{\mathcal P}'(\tilde h_1))g\|_{C^{0, 2+\alpha}} \leq C \|\tilde h_1 - \tilde h_2\|_{C^{0, 2+\alpha}}
$$
from \eqref{linear-P} and Lemma \ref{lem f c 0 parameter dependence}, we obtain again from the
estimates based on Theorem 5.1.21 in \cite{Lu} that
$$
\|v\|_{C^{1, 4+\alpha}} \leq C \|\tilde h_1 - \tilde h_2\|_{C^{0, 2+\alpha}}.
$$
Hence, in the light of Lemma \ref{lem interpolation}, we have
$$
\|v(t_1) - v(t_2) \|_{C^{2, \alpha}} \leq C\cdot |t_1-t_2|\cdot \|\tilde h_1 - \tilde h_2\|_{C^{0, 2+\alpha}}.
$$
In particular
$$
\|\Psi(\tilde h_1) - \Psi(\tilde h_2)\|_{C^{0, 2+\alpha}} \leq CT^*\|\tilde h_1 - \tilde h_2\|_{C^{0, 2+\alpha}}.
$$

To apply contractive mapping theorem we observe that
$$
\|\Psi(\tilde h)\|_{C^{0, 2+\alpha}} \leq \|\Psi(0)\|_{C^{0, 2+\alpha}} + CT^*\|\tilde h\|_{C^{0, 2+\alpha}},
$$
where
$$
\|\Psi(0)\|_{C^{1, 2+\alpha}} \leq C_0 \|\gamma\|_{C^{0, \alpha}},
$$
for some constant $C_0$, from the estimates based on Theorem 5.1.21 in \cite{Lu}. Thus
$$
\Psi: B_R = \{h\in E_1([0, T^*]): \|h\|_{C^{0, 2+\alpha}}\leq R\} \to B_R
$$
for $R=2C_0\|\gamma\|_{C^{0, \alpha}}$ is a contractive mapping when $T^*$ is appropriately small. Then, by the uniqueness of the solution to linear parabolic equations \eqref{eq linear dtcrf-2}, one may extend the solution to $[0, T]$ for \eqref{eq linear dtcrf-2} by steps in time of length $T^*$. The estimate \eqref{est of sol in haf} then follows from the estimates based on Theorem 5.1.21 in \cite{Lu}.
\end{proof}

To summarize we have established that
$$
D{\mathcal M}(g): C^{1, 2+\alpha}([0, T]\times{\rm M}) \bigcap\{h(0, \cdot) = 0\} \to C^{0, \alpha}([0, T]\times{\rm M})
$$
is an isomorphism, provided that $g(t)$ satisfies the assumptions in Proposition \ref{lem 4}.


\subsubsection{Implicit Function Theorem Argument}\label{ex on closed mfld}

Next we solve DeTurck conformal Ricci flow and then conformal Ricci flow. Our approach is to use an implicit function theorem. Let us start with the following simple implicit function theorem.

\begin{lemma}\label{implicit} Suppose $X$ and $Y$ are Banach spaces and
$$
\mathcal {H}: X\to Y
$$
is a $C^1$ map. Suppose that there is a point $x_0\in X$ such that
there are positive numbers $\delta$ and $C$, and
$$
\|(D\mathcal {H}(x))^{-1}\| \leq C \quad \forall \quad x \in B_\delta (x_0)
$$
and
$$
\|D\mathcal {H}(x_1) - D\mathcal {H}(x_2)\| \leq \frac 1{2C}, \quad \forall \quad x_1, x_2
\in B_\delta (x_0).
$$
Then, if
$$
\|\mathcal {H}(x_0)\| \leq \frac \delta{2C},
$$
then there is $x\in B_\delta (x_0)$ such that
$$
\mathcal {H}(x) = 0.
$$
\end{lemma}

To apply the above implicit function theorem to the map for solving DeTurck conformal Ricci flow
$$
\mathcal {M}: C^{1, 2+\alpha}([0, T]\times {\rm M})\bigcap\{g(0) = g_0\} \to C^{0, \alpha}([0, T]\times {\rm M})
$$
we need to show that $\mathcal {M}$ is continuously differentiable.  In fact we have

\begin{lemma}\label{continuity-DM} Suppose that ${\rm M}^m$ is a compact manifold without boundary. Suppose that  $g(t)\in C^{1, 2+\alpha}([0, T]\times {\rm M})$ such that the elliptic operator
$(m-1)\Delta[g(t)] +s_0$ is invertible for all $t\in [0, T]$. Then there is $\delta_0>0$ such that
$$
\|D\mathcal {M}(g_1) - D\mathcal {M}(g_2)\|_{L(C^{1, 2+\alpha}, C^{0, \alpha})} \leq C\|g_1 - g_2\|_{C^{1, 2+\alpha}}
$$
for $\|g_i -  g\|_{C^{1, 2+\alpha}} \leq \delta_0$ in $C^{1, 2+\alpha}([0, T]\times {\rm M})$ and $i=1,2$.
\end{lemma}

\begin{proof} We calculate, for any $h\in C^{1, 2+\alpha}\bigcap\{h(0, \cdot) = 0\}$,
$$
\aligned
(D\mathcal {M}(g_1) - D\mathcal {M}(g_2))h & = (\Delta[g_2] - \Delta[g_1])h
+ 2 \mathcal {P}'[g_1](g_1 - g_2)\\  + 2(\mathcal {P}'[g_1] - \mathcal {P}'[g_2])g_2  &
+ M_1[g_1]\ast (\nabla[g_1] h  - \nabla[g_2] h)  \\ + (M_1 [g_1]- M_1[g_2])\ast \nabla[g_2] h
& + (M_2 [g_1]- M_2[g_2])\ast h.
\endaligned
$$
It is easily seen that
$$
\| \Delta [g_1] h - \Delta [g_2]h\|_{C^{0, \alpha}} \leq C \|g_1 - g_2\|_{C^{1,  2+\alpha}}\|h\|_{C^{1, 2+\alpha}}
$$
$$
\aligned
\|M_1[g_1]\ast (\nabla[g_1] h  - \nabla[g_2] h)  & + (M_1 [g_1]- M_1[g_2])\ast \nabla[g_2] h\|_{C^{0, \alpha}} \\
\leq  & C\|g_1 - g_2\|_{C^{1, 2+\alpha}} \|h\|_{C^{1, 2+\alpha}}
\endaligned
$$
and
$$
\|(M_2[g_1] - M_2[g_2])h\|_{C^{0, \alpha}}\leq C \|g_1 - g_2\|_{C^{1, 2+\alpha}} \|h\|_{C^{1, 2+\alpha}}.
$$
It is also easy to see that
$$
\|\mathcal {P}'[g_1](g_1 -g_2)\|_{C^{0, \alpha}} \leq C\|g_1 - g_2\|_{C^{1, 2+\alpha}} \|h\|_{C^{1, 2+\alpha}}.
$$
under the assumption that $\|g_i - g\|_{C^{1, 2+\alpha}}\leq \delta_0$ for $i=1,2$ from the definition of ${\mathcal P}'$ in \eqref{linear-P}.  For the last remaining term we write
$$
\aligned
((m-1)\Delta +s_0) & (\mathcal {P}'[g_2] - \mathcal {P}'[g_1]) = (m-1)(\Delta[g_1] - \Delta[g_2])\mathcal {P}'[g_1] \\
+ (P_1[g_2]  & - P_1[g_1]) \ast\nabla^2[g_2] h + P_1[g_1]\ast(\nabla^2  [g_2]- \nabla^2[g_1])h\\
+ (P_2[g_2]  & -  P_2[g_1])\ast \nabla[g_2] h + P_2[g_1]\ast (\nabla[g_2] - \nabla [g_1])h \\
+ (P_3[g_2]  & - P_3[g_1])\ast h
\endaligned
$$
and apply Lemma \ref{lem f c 0 parameter dependence}. Therefore
$$
\|\mathcal {P}'[g_1] - \mathcal {P}'[g_2]\|_{C^{0, \alpha}} \leq C \|g_1 - g_2\|_{C^{1, 2+\alpha}} \|h\|_{C^{1, 2+\alpha}},
$$
which implies
$$
\|(\mathcal {P}'[g_1] - \mathcal {P}' [g_2])g\|_{C^{0, \alpha}} \leq C \|g_1 - g_2\|_{C^{1, 2+\alpha}} \|h\|_{C^{1, 2+\alpha}}.
$$
Thus the proof is complete.
\end{proof}

Next, to apply Lemma \ref{implicit}, we consider the initial approximate solution as follows:
\begin{equation}\label{x_0}
\bar g(t) = g_0 + t \mathcal {F}(g_0)
\end{equation}
where $\mathcal {F}$ is introduced in \eqref{set-up}. We then calculate
\begin{equation}\label{H(x_0)}
\aligned
\mathcal {M}(\bar g) & = - \mathcal {F}(g_0 + t\mathcal {F}(g_0))  + \mathcal {F}(g_0)\\
& = - t \int_0^1D\mathcal {F} (g_0 + \theta t \mathcal {F}(g_0))d\theta\cdot \mathcal {F}(g_0).
\endaligned
\end{equation}
Now we are ready to state and prove the short time existence theorem for conformal Ricci flows.

\begin{theorem}\label{short time existence on cpt} Let ${\rm M}^m$ be a compact manifold with no boundary. Suppose that $g_0\in C^{4, \alpha}$ is a Riemannian metric on ${\rm M}$ such that the scalar curvature $s[g_0] = s_0$ is a constant and that the elliptic operator $(m-1)\Delta[g_0] + s_0$ is invertible. Then there exists a small positive number $T$ such that the conformal Ricci flow $g(t)$ exists in $C^{1, 2+\alpha}$ from the initial metric $g_0$ for $t\in [0, T]$.
\end{theorem}

\begin{proof} First we notice that Proposition \ref{lem 4} holds for the family of metrics $\bar g(t) =g_0 + t\mathcal{F}(g_0)$ in $C^{1, 2+\alpha}$ for some appropriately small $T$ such that the elliptic operator $(m-1)\Delta[\bar g] +s_0$ is invertible for all $t\in [0, T]$. Therefore there is a constant $C$ and a small number $\delta_0$ such that
$$
\|(D\mathcal {M} (g))^{-1}\|\leq C
$$
and
$$
\|D\mathcal {M} (g_1) - D\mathcal {M}(g_2)\|\leq \frac 1{2C}
$$
for all $g, g_1, g_2, \in B(\delta_0)$, where $B(\delta_0) = \{g\in C^{1, 2+\alpha}: \|g -\bar g\|_{C^{1, 2+\alpha}} \leq \delta_0\}$, according to Lemma \ref{continuity-DM}. Next, choose even smaller $T$ if necessary, we observe from \eqref{H(x_0)} that
$$
\|\mathcal {M}(\bar g)\|_{C^{0, \alpha}} \leq \frac {\delta_0} {2C}.
$$
Hence Lemma \ref{implicit} implies that DeTurck conformal Ricci flow $\hat g(t)$ exists in $C^{1, 2+\alpha}$ with the initial metric $g_0$. Therefore, applying Lemma \ref{lem dcrf to sol crf},
we obtain the short time existence for conformal Ricci flow from the initial metric $g_0$, since \eqref{ode} is always solvable at least for short time.
\end{proof}


\subsection{Short Time Existence on Asymptotically Flat Manifolds}

In this subsection we establish the short time existence of conformal Ricci flow on asymptotically flat manifolds.
The idea of the proof is the same as the proof in last subsection.  We remark here that the short time existence of Ricci flow on asymptotically flat manifolds has been established independently in \cite{DM} and \cite{OW}. The approach in \cite{DM} is to use the short time existence result in
\cite{Shi} and the maximum principle to show that Ricci flow in fact remains to be asymptotically flat when starting from asymptotically flat metric; while the approach in \cite{OW} is to establish
short time existence of Ricci flow based on weighted function spaces. Our approach is similar to the one in \cite{OW} since no short time existence of conformal Ricci flow on non-compact manifolds is available and no maximum principle for conformal Ricci flow is available.


\subsubsection{Analysis on Asymptotically Flat Manifolds}

We first briefly introduce asymptotically flat manifolds according to \cite{LP} and then construct
appropriate parabolic H\"{o}lder spaces on asymptotically flat manifolds. The following is from  \cite[Definition 6.3]{LP}.

\begin{definition}
A Riemannian manifold ${\rm M}^m$ with $C^2$-metric $g$  is called
asymptotically flat of order $\tau >0$
if there exists a decomposition ${\rm M}={\rm M}_0 \cup {\rm M}_{\infty}$
(with ${\rm M}_0$ compact) and a diffeomorphism $\Psi: {\rm M}_{\infty} \rightarrow \mathbb{R}^n\setminus B_{\rm R}(\vec{0})$, for some ${\rm R}>0$,
satisfying:
$$
g (z)=g_e(z) + O(\rho^{-\tau}),
 \quad \partial_k g (z)=O(\rho^{-\tau-1}),
 \quad \partial_k \partial_l g(z)=O(\rho ^{-\tau-2}),
$$
where $g_e$ is the standard Euclidean metric and $\rho=\rho(z) =\vert z \vert \to \infty$ in the coordinates $z= (z^1,\cdots,z^m)$ induced on $M_{\infty}$ by the diffeomorphism $\Psi$.
\end{definition}

We adopt the definition of weighted H\"{o}lder spaces $C^{k, \alpha}_\beta$ from \cite[p. 75]{LP}. Again we will use the same notations for weighted H\"{o}lder spaces for functions and tensor fields if there is no confusion. We use local coordinate charts and a given metric whenever it is necessary for the definition of H\"{o}lder spaces for tensor fields on asymptotically flat manifolds.

Fix a number $T>0$, analogous to \cite[p. 175-177]{Lu},  we define parabolic weighted H\"{o}lder spaces
$$
C^{0, k+ \alpha}_{\beta}
: = \{ h\in C([0, T]\times{\rm M}): h(t) \in C^{k, \alpha}_\beta \ \text{and} \ \max_{t \in [0,T]}
 \| h(t) \|_{C^{k, \alpha}_{\beta }} < \infty\}
$$
with the norm
$$
\|h\|_{C^{0, k+ \alpha}_{\beta}} : = \max_{t \in [0,T]}
 \| h(t) \|_{C^{k, \alpha}_{\beta }}.
$$
Similarly we define
$$
C^{1, k+ \alpha}_{\beta} : = \{h\in C^{0, k+\alpha}_\beta \ \text{and} \ \partial_t h\in C^{0, k-2+\alpha}_{\beta -2}\}
$$
with the norm
$$
\|h\|_{C^{1, k + \alpha}_{\beta} } : = \max_{t \in [0,T]}
 \| h(t) \|_{C^{k, \alpha}_{\beta }} + \max_{t \in [0,T]}
 \| \partial_t h(t) \|_{C^{k-2, \alpha}_{\beta -2}}.
$$
We now recall the elliptic theory for weighted H\"{o}lder spaces, for example, from \cite[Theorem 9.2]{LP} in our context.

\begin{lemma}\label{Laplace invertible weighted}
Let $({\rm M}^m, \ g(t))$, for $t\in [0, T]$,  be a family of asymptotically flat manifolds with $g(t)-g_e \in C^{0, 2+\alpha}_{-\tau}$ for $\tau >0$. Then for $\beta \in (2-m, 0)$
\[
\Delta[g(t)]: C_\beta^{0, 2+\alpha} \rightarrow C_{\beta-2}^{0, \alpha}
\]
is an isomorphism, that is, there is $C$ such that
$$
\|u\|_{C^{0, 2+\alpha}_\beta} \leq C \|\Delta[g(t)] u\|_{C^{0, \alpha}_{\beta-2}}.
$$
\end{lemma}

Analogous to Lemma \ref{lem interpolation} we have a simple interpolatory inclusion.

\begin{lemma} \label{lem interpolation weighted} There is a constant $C$
independent of $T$ such that for any $t_1, t_2 \in [0,T]$ we have
\[
\|h(\cdot,t_1) - h(\cdot,t_2)\|_{C^{k-2, \alpha}_{\beta-2}} \leq C
\cdot |t_1 -t_2 |\cdot \| h\|_{C^{1, k+\alpha}_\beta}.
\]
for all $h \in C^{1, k + \alpha}_\beta([0, T]\times {\rm M})$.
\end{lemma}


\subsubsection{Short Time Existence on Asymptotically Flat Manifolds}

In this subsection we assume that the initial metric $g_0$ on ${\rm M}^m$ is asymptotically
flat and scalar flat.  Thanks to  \cite[Lemma 3.3, Corollary 3.1]{SY} we know that one can always conformally deform  an asymptotically flat metric with nonnegative scalar curvature into a scalar flat asymptotically flat metric. We will use the same strategy as in Section \ref{subsec Holder space def} to prove the short time existence of conformal Ricci flow on asymptotically flat manifolds.

First with changes of notations we are able to prove the isomorphism analogous to Proposition \ref{lem 4}. An extension of  \cite[Theorem 5.1.21]{Lu} to the weigted parabolic H\"{o}lder spaces on asymptotically flat manifolds may be proven by the standard argument through
interior estimates and scaling invariance of the interior estimates (cf. \cite{OW} \cite{Ba} \cite{LP}). The key is to realize that one may move in and out the weight for local estimates.

\begin{proposition}\label{weighted lem 4}
 Suppose that $g(t)$, $t\in [0, T_0]$,  is a family of asymptotically flat metrics with $g(t) - g_e\in C^{0, 2+\alpha}_{-\tau}$ with $\tau \in (0, m-2)$.  Then there is a $T_* \in (0, T_0]$ such that, for any $T\leq T_*$ and $\gamma \in C^{0, \alpha}_{-\tau - 2}$, the initial value problem for \eqref{eq linear dtcrf-2} has a unique solution $h \in C^{1, 2+\alpha}_{-\tau}$. Moreover
$$
 \|h\|_{C^{1, 2+\alpha}_{-\tau}([0, T]\times{\rm M})}
\leq  C \| \gamma \|_{ C^{0, \alpha}_{-\tau-2}([0, T]\times{\rm M})}.
$$
 \end{proposition}

This is to say that, for $\tau\in (0, m-2)$,
$$
D{\mathcal M}(g): C^{1, 2+\alpha}_{-\tau}([0, T]\times{\rm M}) \bigcap\{h(0, \cdot) = 0\} \to C^{0, \alpha}_{-\tau -2}([0, T]\times{\rm M})
$$
is an isomorphism, provided that $g(t)$ and $T$ satisfiy the assumptions in the above Proposition \ref{weighted lem 4}. The restriction on the order $\tau$ of weight is solely used in solving elliptic equations on weighted spaces in Lemma \ref{Laplace invertible weighted}.

To obtain a short time existence of DeTurck conformal Ricci flow we again apply the implicit function theorem (Lemma \ref{implicit}) to the map
$$
{\mathcal M}: \{g(t): g(t) - g_e\in C^{1, 2+\alpha}_{-\tau}([0, T]\times{\rm M}) \ \text{and} \ g(0) = g_0\} \to C^{0, \alpha}_{-\tau -2}([0, T]\times{\rm M})
$$
for any $\tau\in (0, m-2)$ and $T$ given from Proposition \ref{weighted lem 4}. Finally we arrive at the short time existence of conformal Ricci flow.

\begin{theorem}\label{short time existence on af} Let $({\rm M}^m, \ g_0)$ be scalar flat and asymptotically flat with $g_0 - g_e\in C^{4, \alpha}_{-\tau}$ and $\tau \in (0, m-2)$.  Then there exists a small positive number $T$ such that the conformal Ricci flow $g(t)$  from the initial metric $g_0$ exists for $t\in [0, T]$ and $g(t)-g_e \in C^{1, 2+\alpha}_{-\tau}([0, T]\times{\rm M})$.
\end{theorem}

\begin{proof}
As in Section \ref{ex on closed mfld} we first verify
$$
\|D{\mathcal M}(g_1) - D{\mathcal M}(g_2)\|_{L(C^{1, 2+\alpha}_{-\tau}, C^{0, \alpha}_{-\tau-2})}
\leq C \|g_1 - g_2\|_{C^{1, 2+\alpha}_{-\tau}}.
$$
The proof goes like the one for Lemma \ref{continuity-DM} with only changes of notations.
We then construct
$$
\bar g(t) = g_0 + t{\mathcal F}(g_0)\in g_e + C^{1, 2+\alpha}_{-\tau}
$$
as in Section \ref{ex on closed mfld}, for $g_0 - g_e\in C^{4, \alpha}_{-\tau}$. Another issue one needs to take care of is solving \eqref{ode} to construct conformal Ricci flow from DeTurck conformal Ricci flow. But, since $W\in C^{0, 1+\alpha}_{-\tau -1}$, it is easy to solve \eqref{ode} on the whole manifold ${\rm M}$ for some short time. The rest of the proof goes like the one in Section \ref{ex on closed mfld} for Theorem \ref{short time existence on cpt} with little changes except notations. Notice that the equivalence between \eqref{eq crf orig-0} and \eqref{eq crf orig} holds because of the uniqueness of the bounded solution to linear parabolic equations on asymptotically flat manifolds.
\end{proof}


\section{ADM Mass under Conformal Ricci Flow}\label{adm mass monotone}

Asymptotically flat manifolds are used in general relativity to describe isolated gravitational systems. The fundamental geometric invariant of a asymptotically flat manifold is called the mass of the gravitational system. The so-called ADM mass of an asymptotically flat manifold was first defined in \cite{ADM} in early 1960s.

In general relativity the world is modeled by a 4-dimensional spacetime ${\rm X}^4$ with a Lorentzian metric $g$. The physical law that describes the gravity induced by matters in the spacetime is the famous Einstein equations
$$
Ric[g] - \frac 12 s[g]g = T,
$$
where $T$ is the energy-momentum-stress tensor that is supposed to reflect the nature and state of the matters in the spacetime. A time slice of a space-time that represents an isolated gravitational system is an asymptotically flat  3-manifold ${\rm M}^3$.

One of the most important solution of Einstein equations is the Schwarzchild spacetime, which represents the gravitational system of a static point particle of mass $m$ and whose a time slice is an asymptotically flat metric
$$
g_{Sch} = g_e + \frac m\rho g_e + O(\rho^{-2})
$$
on the punctured $\mathbb{R}^3$. The crucial test to validate the notion of mass in relativity is that its prediction reduces to those of Newtonian gravity under the circumstances when Newtonian theory is known to be valid - specifically, when gravity is weak, motions are much slower than the speed of the light, and material stresses are much smaller than the mass-energy density (cf. \cite[4.4]{Wald}).

We now follow \cite[Definition 8.2]{LP} to introduce ADM mass for asymptotically flat manifolds.

\begin{definition}\label{ADM mass}  Given an asymptotically flat Riemannian
manifold $({\rm M}^m, g)$ with asymptotic coordinates $z$, we define the ADM mass by
(if the limit exists)
\[
m(g)=\lim_{R \to \infty} \omega_{m-1} ^{-1}\int_{\mathbb{S}_R}
(\partial_ig_{ij}-\partial_jg_{ii}) n^j d\sigma
\]
where $\omega_{m-1}$ is the volume of the unit sphere $\mathbb{S}^{m-1}$, $\vec{n} =(n^1, \cdots, n^m)$ is the outward unit normal vector of the sphere $\mathbb{S}_R=\{z \in \mathbb{R}^m, \vert z \vert =R\}$ and $d\sigma$ is the area element of $\mathbb{S}_R$.
\end{definition}

Recall from \cite{LP} that
\[
\mathcal {M}_\tau :=  \{g= g_e +h: h \in C^{1, \alpha}_{-\tau}  \text{ and }
\partial_j \partial_i h_{ij} -\partial_j \partial_j h_{ii} \in L^1({\rm M}, d\text{vol}[g_e])  \}.
\]
After Definition \ref{ADM mass} one wonders if the ADM mass is indeed a geometric invariant for the asymptotically flat metric. It was confirmed as follows:

\begin{lemma} (\cite{ADM}\cite{Ba}) Suppose that $g$ is an asymptotically flat metric in $\mathcal {M}_\tau$ on ${\rm M}^m$ for $\tau > \frac {m-2}2$. Then the ADM mass $m(g)$ is indeed independent of the choice of asymptotic coordinates at infinity.
\end{lemma}

Another important fact about the ADM mass is observed in \cite[(8.11)]{LP} and supported by \cite[Lemma 9.4]{LP}.

\begin{lemma}\label{thm lee parker} (\cite{LP})
Let $g(t)$ be a smooth family of asymptotically flat metrics in $\mathcal {M}_\tau$ on ${\rm M}^m$ for $\tau > \frac {m-2}2$. Then the mass $m(g(t))$ is differentiable and
\begin{equation*}
\frac{d}{dt} \left (- \int_{\rm M} s[g(t)]d\text{vol}[g(t)] + \omega_{m-1}m(g(t)) \right ) =
\int_M G[g(t)]\cdot \phi(t) d\text{vol}[g(t)]
\end{equation*}
where  $G[g(t)] = Ric[g(t)]- \frac{1}{2} s[g(t)]g(t)$ is the Einstein tensor and $\phi (t) = \partial_t g(t)$.
\end{lemma}

Consequently we have, from Theorem \ref{short time existence on af} and Lemma \ref{thm lee parker},

\begin{theorem}\label{thm monotone of ADM mass}
Let $g_0$ be a scalar flat and asymptotically flat metric on ${\rm M}^m$ such that
$g_0 - g_e \in C^{4, \alpha}_{-\tau}$ for $\tau \in (\frac {m-2}2, m-2)$.
Then the conformal Ricci flow $g(t)$ starting with $g(0)=g_0$ exists for some short time and
$$
g(t) \in {\mathcal  M}_{\tau} \quad \text{and} \quad g(t) - g_e \in C^{1, 2+\alpha}_{-\tau}.
$$
Moreover
\[
\frac{d}{dt} m(g(t)) = -2 \int_M |Ric[g(t)] |^2 d\text{vol}[g(t)].
\]
In particular, the ADM mass is strictly decreasing under conformal Ricci flow except that $g_0$ is
the Euclidean metrics.
\end{theorem}

\begin{proof} To verify that conformal Ricci flow $g(t)\in\mathcal {M}_\tau$ we only need to verify that
$$
\partial_j\partial_i g_{ij} (t) - \partial_j\partial_j g_{ii}(t)\in L^1({\rm M}, d\text{vol}[g_e]).
$$
Recall \cite[(9.2)]{LP}
$$
s = \partial_j\partial_i g_{ij} - \partial_j\partial_j g_{ii} + O(\rho^{-2\tau - 2})
$$
which implies that
$$
\partial_j\partial_i g_{ij} - \partial_j\partial_j g_{ii} = O(\rho^{-2\tau -2})\in L^1({\rm M},  d\text{vol}[g_e]).
$$
for $\tau \in (\frac {m-2}2, m-2)$. It is easily seen that the ADM mass is strictly decreasing except that $g_0$ is Ricci flat. Then, using \cite[Theorem 1.5]{BKN} and \cite[Proposition 10.2]{LP}, one concludes that $g_0$ is the standard Euclidean metric. Therefore the proof is complete.
\end{proof}

A quick application of the above Theorem \ref{thm monotone of ADM mass} is a simple and direct proof of the rigidity part of Schoen and Yau positive mass theorem. Namely,

\begin{corollary} (\cite{SY}) Suppose that $\left({\rm M}^m, \ g\right)$ is asymptotically flat manifold with nonnegative scalar curvature and that $g - g_e \in C^{4, \alpha}_{-\tau}$
for $\tau > \frac {m-2}2$. Then, if the ADM mass $m(g) = 0$, then $({\rm M}^m, \ g)$ is isometric to the standard Euclidean space $\mathbb{R}^m$.
\end{corollary}

\begin{proof} First we know $g$ has to be scalar flat. Otherwise one can conformally deform the metric to scalar flat and decrease the ADM mass to negative, which is impossible
due to the first part of the positive mass theorem of Schoen and Yau. Next we invoke Theorem \ref{thm monotone of ADM mass} and come to the same contradiction if $g$ is not flat.
\end{proof}


\bibliographystyle{natbib}

\end{document}